\documentclass[12pt]{amsart}

\PassOptionsToPackage{table}{xcolor}
\usepackage[table]{xcolor}

\usepackage[T1]{fontenc}
\usepackage{amsmath}
\usepackage{amssymb}
\usepackage{tikz-cd}
\usepackage{multirow}

\usepackage[backend=biber]{biblatex}
\newtheorem{theorem}[subsection]{Theorem}
\newtheorem{proposition}[subsection]{Proposition}
\newtheorem{corollary}[subsection]{Corollary}
\newtheorem{lemma}[subsection]{Lemma}

\theoremstyle{definition}
\newtheorem{definition}[subsection]{Definition}

\newcommand{\Z}{{\mathbb Z}}
\newcommand{\R}{\mathbb{R}}
\renewcommand{\P}{{\mathbb P}}
\newcommand{\Q}{{\mathbb Q}}

\title{Bousfield-Kan Completion and Very Large Groups}
\author{ Jaime Benabent Guerrero, Ramón Flores}
\date{\today}

\makeatletter
\def\@tocline#1#2#3#4#5#6#7{\relax
  \ifnum #1>\c@tocdepth 
  \else
    \par \addpenalty\@secpenalty\addvspace{#2}%
    \begingroup \hyphenpenalty\@M
    \@ifempty{#4}{%
      \@tempdima\csname r@tocindent\number#1\endcsname\relax
    }{%
      \@tempdima#4\relax
    }%
    \parindent\z@ \leftskip#3\relax \advance\leftskip\@tempdima\relax
    \rightskip\@pnumwidth plus4em \parfillskip-\@pnumwidth
    #5\leavevmode\hskip-\@tempdima
      \ifcase #1
       \or\or \hskip 1em \or \hskip 2em \else \hskip 3em \fi%
      #6\nobreak\relax
    \hfill\hbox to\@pnumwidth{\@tocpagenum{#7}}\par
    \nobreak
    \endgroup
  \fi}
\makeatother

\begin{document}
	
	\begin{abstract}
		We establish new criteria for the $R$-badness of a space and apply it to the case of closed surfaces.
	\end{abstract}
	
	\maketitle

    \section{Introduction}
    \noindent
    


The Bousfield–Kan $R$-completion, developed in \cite{BK72}, provides a unifying framework for understanding how spaces behave with respect to a chosen ring of coefficients $R$, generalizing both $p$-completion and rationalization within a single homotopical setting. When $R =\mathbb{Z}/p$, the $R$-completion isolates the $p$-primary information of a space, capturing precisely those features detected by mod-$p$ (co)homology. This construction has played a fundamental role in homotopy theory, notably in the study of classifying spaces of finite and profinite groups, in the analysis of $p$-local homotopy types, and in the characterization of maps inducing mod-$p$ homology equivalences. It is also central to the proof of the Sullivan conjecture and to the homotopy classification of finite $H$-spaces. Moreover, the theory provides a geometric analogue of $p$-adic completion for groups, linking the $R$-completion of a space with the $p$-adic completion of its fundamental group in the nilpotent case.

When $R$ is a subring of $\mathbb{Q}$, such as $\mathbb{Z}[1/p]$ or $\mathbb{Q}$ itself, $R$-completion is related to rational localization. In these cases, the process annihilates torsion phenomena invisible to $R$-homology and produces a more algebraic homotopy type that can often be described via differential graded Lie algebras or minimal models. This rational completion framework underlies rational homotopy theory, where the $\mathbb{Q}$-completion of a simply connected space encodes its entire homotopy type in purely algebraic terms. Furthermore, $R$-completion for subrings of $\mathbb{Q}$ provides the essential bridge in the arithmetic fracture square, connecting rational and $p$-adic information and allowing global homotopy types to be reconstructed from their local components.

In a broader sense, the $R$-completion formalism offers a conceptual mechanism for studying localizing phenomena simultaneously across arithmetic and homological contexts. It has profound applications to mapping spaces, localizations of ring spectra, and the interplay between algebraic and topological completions. Whether isolating $p$-primary structure or extracting rational information, $R$-completion remains one of the most powerful tools for decomposing complex homotopy types into algebraically manageable local parts.

However, unlike Bousfield homological localization, the Bousfield–Kan completion functor is not always idempotent on spaces; that is, applying it twice does not necessarily yield the same result as applying it once. From the beginning, Bousfield-Kan sought conditions under which the completion functor becomes idempotent, and many of these are listed in their foundational monograph. It is precisely on such ``good” spaces that completion behaves predictably. Yet there remain many spaces for which idempotency is not known, and explicit examples of spaces that fail it —the so-called $R$-bad spaces— are rare and of particular interest.

This paper addresses that gap by establishing a general criterion for detecting $R$-badness. Our main result (Theorem \ref{th:main}) states that if $X$ is a connected space whose second integral homology group $H_2(X;\mathbb{Z})$ is countable and whose fundamental group admits a surjection onto a non-commutative free group (a \emph{very large} group, see Definition \ref{Def:very large}), then $X$ is $R$-bad for any $R=\mathbb{Z}/p$ or any subring $R\subseteq \mathbb{Q}$. In particular, the classifying spaces of these groups are $R$-bad. This criterion, which uses and generalizes recent work of Ivanov-Mikhailov \cite{IM19} showing that the wedge of two circles is $R$-bad, provides a systematic framework for producing new examples of non-idempotent completions.

Our initial motivation arose from a question posed by A. Murillo \cite{Mur20} concerning whether compact orientable surfaces are ``good” spaces with respect to rational completion. Surprisingly, even for such elementary geometric objects, only the sphere and the torus were known to be $R$-good, while the cases of bigger genus remained open. By observing that the fundamental groups of compact surfaces (with very few exceptions) are very large, and by combining this fact with the aforementioned badness criterion, we completely resolve Murillo’s question. We prove that all compact orientable surfaces other than the sphere and the torus are $R$-bad for $R=\mathbb{Z}/p$ and for subrings of $\mathbb{Q}$, and that the same holds for all compact non-orientable surfaces of genus bigger than 3; the cases of genus 1 and 2 where previously known, while our methods do not fit in the case of genus 3, which remains unsolved. Moreover, our methods extend to a wide class of very large groups satisfying mild finiteness assumptions, thus offering a flexible and general approach to producing new examples of $R$-bad spaces.

The paper is organized as follows. In Section 2, we recall the main definitions and results concerning Bousfield–Kan $R$-completion, together with a review of known examples of $R$-good and $R$-bad spaces. In Section 3 we prove our main theorem, establishing the criterion for $R$-badness, while in Section 4 this criterion is applied to the case of closed surfaces. Finally, Section 5 extends the arguments to broader families of very large groups such as general Artin groups, right-angled Artin groups, Bestvina–Brady groups, and graph braid groups, thereby illustrating the scope and utility of our criterion.

	\section{Bousfield-Kan $R$-Completion}
	\label{background}
	\noindent


	In this section, we recall the main definitions and results concerning the Bousfield-Kan $R$-completion and compile some currently known examples of $R$-good and $R$-bad spaces for different choices of a ring $R$.
	
	\subsection{$R$-completion}
    
	In \cite{BK72} is introduced, for any ring $R$, a $R$-\emph{completion} in the category of spaces, i.e. an endofunctor $R_{\infty} : \mathcal{S} \to \mathcal{S}$ together with a natural transformation $\epsilon: 1_\mathcal{S} \to R_{\infty}$, called the coaugmentation, such that $R_{\infty}\epsilon_X = \epsilon_{R_{\infty}}$. This functor is defined, up to homotopy, for the following property:
	
	\begin{proposition}
		A morphism of spaces, $f: X \to Y$, is a $R$-homology equivalence if and only if the morphism given by the $R$-completion, $R_{\infty}f : R_{\infty}X \rightarrow R_{\infty}Y$, is a homotopy equivalence.
	\end{proposition}
	
 	This construction, which is performed by means of an inverse homotopy limit of a tower of spaces, has an important disadvantage with regard to its respective homological localization (see \cite{Bou75}): it is not always idempotent, i.e. in general $R_{\infty}X$ is not always homotopy equivalent to $R_{\infty}R_{\infty}X$.
    This motivated the following definition, with its corresponding proposition:
	
	\begin{definition}
		Let $X$ be a space.
		\begin{enumerate}
			\item $X$ is called $R$-\emph{complete} if $\epsilon_X$ is a homotopy equivalence.
			
			\item $X$ is called $R$-\emph{good} if $\epsilon_X$ is an $R$-homology equivalence.
			
			\item $X$ is called $R$-\emph{bad} if it is not $R$-good.
		\end{enumerate}
	\end{definition}
	
	\begin{proposition}
		Let $X$ be a space.
		The following conditions are equivalent:
		\begin{itemize}
			\item $X$ is $R$-good.
			\item $R_{\infty}X$ is $R$-good.
			\item $R_{\infty}X$ is $R$-complete.
		\end{itemize}
	\end{proposition}
	
	In particular, if we define inductively $R^0_{\infty} := 1_\mathcal{S}$ and $R^n_{\infty} := R_{\infty} \circ R^{n-1}_{\infty}$ for $n \geq 1$, for every space we either have an infinite chain
	$$
	R_{\infty}X \stackrel{\simeq}{\to}
	R^2_{\infty}X \stackrel{\simeq}{\to}
	R^3_{\infty}X \stackrel{\simeq}{\to} \cdots
	$$
	of homotopy equivalences or we never have a homotopy equivalence in the infinite chain
	$$
	X \stackrel{\not\simeq}{\to}
	R_{\infty}X \stackrel{\not\simeq}{\to}
	R^2_{\infty}X \stackrel{\not\simeq}{\to}
	R^3_{\infty}X \stackrel{\not\simeq}{\to} \cdots.
	$$
	
	\subsection{The choice of the ring}. Unless explicit mention against, the result of this paper are valid when $R=\mathbb{Z}/p$ for $p$ prime, or $R$ is a subring of the rationals. The reason of this is that $R$-completion is preserved when we change $R$ by its \emph{core}, i.e. the maximal solid ring of $R$. Recall that a ring $R$ is called \emph{solid} if for any other ring $S$ there exists at most one morphism $R \to S$. Then, the classification of the solid rings implies that basically the cases of interest from the point of view of $R$-completion are  $R=\mathbb{Z}/p$ for $p$ prime, or $R$ is a subring of the rationals. See \cite[Ch. I, 9.5]{BK72} for a discussion of the subject, and \cite{Ben25} for a modern approach that contains a complete classification of the solid rings.

    \subsection{The $R$-completion of a group}
    In the proof of the main result we will make use of the $R$-completion of a group, particularly we will deal with the $R$-completion of $F_2$ with respect to $R = \Z/p\Z$ for $p$ prime or $R \subseteq \Q$ a subring of the rationals.

    \begin{definition}
        Let $R$ be a ring.
        A group $G$ is called \textbf{$R$-nilpotent} if $G$ has a finite central series such that the succesive quotients admit an $R$-module structure.
    \end{definition}

Now fix a ring $R$ and a group $G$. Consider the category $\mathcal{N}_G$ whose objects are the homomorphisms $G \to N$, where $N$ is an $R$-nilpotent group, and whose morphisms are given by the commutative triangles $G \to N \to N'$. Consider also the functor $F$ from $\mathcal{N}_G$ to the category of groups that sends $G \to N$ to $N$ and the triangle $G \to N \to N'$ to the corresponding homomorphism $N \to N'$. Then we have the following definition.

    \begin{definition}
        The (group) \textbf{$R$-completion} $\widehat{G}_R$ of $G$ is the inverse limit of the functor $F$.
    \end{definition}

This construction is described in \cite{BK72}, Chapter IV, as well as its relation with the completion of spaces.
    
	\subsection{Criteria for $R$-goodness and $R$-badness}
    In this subsection, we give general criteria for a space to be $R$-good or $R$-bad.
    We start with some definitions.
    
    \begin{definition}
        A group $G$ is called \textbf{$R$-perfect} if $H_1(G; R) = 0$, i.e. if
        $$
        R \otimes G^{ab} = 0.
        $$
    \end{definition}

    \begin{definition}
    \begin{enumerate}
        \item 
        A group is called \textbf{nilpotent} if its $\mathbb{Z}$-nilpotent, i.e. its lower central series stabilizes after a finite length at the trivial subgroup.

        \item
        A group is called \textbf{virtually nilpotent} if it has a nilpotent normal subgroup of finite index.

    \end{enumerate}
    \end{definition}



    For the following definitions we refer to \cite{DDK77}:

    \begin{definition}
    \begin{enumerate}
        \item
        A space is called \textbf{nilpotent} if it is connected and its fundamental group is a nilpotent group whose action on the higher homotopy groups is nilpotent.
        Equivalently, it is a space for which, up to homotopy, the Postnikov tower can be refined to a tower of principal fibrations.

        \item
        A space is called \textbf{virtually nilpotent} if it is connected and its fundamental group is a virtually nilpotent group whose action on the higher homotopy groups is virtually nilpotent.
        Equivalently, it is a space for which each Postnikov stage has a finite covering space which is nilpotent.
    \end{enumerate}
    \end{definition}

    \vspace{1pt}
    \textbf{Criteria for $R$-good spaces}
    \begin{itemize}
        \item[G1] Any space with finite homotopy groups in each dimension is $R$-good for $R \subseteq \Q$ and $R = \Z/p\Z$, as shown in \cite[Ch. VII, 4.3]{BK72}.\\

        \item[G2] Any space whose fundamental group is finite is $R$-good for $R = \Z/p\Z$, as shown in \cite[Ch. VII, 5.1]{BK72}.\\

        \item[G3] Any space with an $R$-perfect fundamental group is $R$-good for $R \subseteq \Q$ and $R = \Z/p\Z$, as shown in \cite[Ch. VII, 3.2]{BK72}.\\
    
        \item[G4] Nilpotent spaces are $R$-good for $R \subseteq \Q$ and $R = \Z/p\Z$, as shown by \cite[Ch. V, 3.4]{BK72} and \cite[Ch. VI, 5.3]{BK72}.\\

        \item[G5] Virtually nilpotent spaces are $R$-good for $R = \Q$ and $R = \Z/p\Z$, as shown in \cite[Proposition 3.4]{DDK77}.\\

        \item[G6] A $p$-seminilpotent space is $\Z/p\Z$-good, as shown in \cite[Theorem 4.3]{Bo92}.\\
    \end{itemize}
    
    \textbf{$R$-bad spaces.}

    Contrary to the previous case, no general criteria can be found in the literature to state that a space is $R$-bad, but only isolated examples, which we next will review. The main result of this paper, which is based in turn in the $R$-badness of the wedge of circunferences $S^1 \vee S^1$, provides one such criterion.

    \subsection{Examples of $R$-good and $R$-bad spaces}
    We now give specific examples of $R$-good and $R$-bad spaces present in the literature.

    \vspace{5pt}
    \textbf{$\bullet$ Wedge of circunferences.}
    It was first shown in \cite[Ch. IV, 5.4]{BK72} that a wedge of a countable number of circunferences, i.e. $K(F, 1)$ where $F$ is a free group in a $\aleph_0$ number of generators, is $R$-bad for $R = \Z$ and $R = \Z/p\Z$.
    
    Subsequent work by Ivanov and Mikhailov, as in Lemma \ref{lmm:aux1} below, would show that a wedge of two circles $S^1 \vee S^1$, is $R$-bad for $R \subseteq \Q$ and $R = \Z/p\Z$.
    The main contribution of this paper is a generalization of this result based on their arguments.
	
	\begin{center}
	\begin{tabular}{|c|c|}
		\hline
		\multicolumn{2}{|c|}
		{\cellcolor[HTML]{CBCEFB}\color[HTML]{333333}$\bigvee\limits_{\alpha} S^1$ s.t. $2 \leq |\alpha| \leq \aleph_0$}\\
		\hline
		{\cellcolor[HTML]{ECF4FF}$R \subseteq \Q$} & {\cellcolor[HTML]{ECF4FF}$R = \Z/p\Z$} \\
		\hline
		{\cellcolor[HTML]{FFCCC9}bad} & {\cellcolor[HTML]{FFCCC9}bad}\\
		\hline
	\end{tabular}
	\end{center}

    \textbf{$\bullet$ The projective plane $R\P^2$.}
    The fundamental group of the projective plane is $\Z/2\Z$.
    Then, we have that for $R = \Z/p\Z$ it is $R$-good by condition G2.
    While for $R \subseteq \Q$ we have two cases:
    If $\frac{1}{2} \in R$ it is $R$-good by condition G3, and, if $\frac{1}{2} \not\in R$ then it is $R$-bad, as shown in \cite[Ch. VII, 5.2]{BK72}.

	\begin{center}
	\begin{tabular}{|c|c|c|}
		\hline
		\multicolumn{3}{|c|}
		{\cellcolor[HTML]{CBCEFB}{$\R P^2$}}\\
		\hline
		\multicolumn{2}{|c|}{\cellcolor[HTML]{ECF4FF}$R \subseteq \Q$} & {\cellcolor[HTML]{ECF4FF}}\\
		\cline{1-2}
		{\cellcolor[HTML]{ECF4FF}$1/2 \in R$} & {\cellcolor[HTML]{ECF4FF}$1/2 \not\in R$} & \multirow{-2}{*}{\cellcolor[HTML]{ECF4FF}$R = \Z/p\Z$} \\
		\hline
		{\cellcolor[HTML]{9AFF99}good} & {\cellcolor[HTML]{FFCCC9}bad} & {\cellcolor[HTML]{9AFF99}good}\\
		\hline
	\end{tabular}
	\end{center}

    \textbf{$\bullet$ Klein bottle.}
    Because it is a virtually nilpotent space it is $\Z/p\Z$-good for every prime $p$, as stated in $G2$.
    It is $\Z[J^{-1}]$-good if $2 \in J$ and $\Z[J^{-1}]$-bad as long as $2 \not\in J$, as shown by \cite[Section 4.4]{Bas03}.

	\begin{center}
	\begin{tabular}{|c|c|c|}
		\hline
		\multicolumn{3}{|c|}
		{\cellcolor[HTML]{CBCEFB}{Klein bottle}}\\
		\hline
		\multicolumn{2}{|c|}{\cellcolor[HTML]{ECF4FF}$R \subseteq \Q$} & {\cellcolor[HTML]{ECF4FF}}\\
		\cline{1-2}
		{\cellcolor[HTML]{ECF4FF}$1/2 \in R$} & {\cellcolor[HTML]{ECF4FF}$1/2 \not\in R$} & \multirow{-2}{*}{\cellcolor[HTML]{ECF4FF}$R = \Z/p\Z$} \\
		\hline
		{\cellcolor[HTML]{9AFF99}good} & {\cellcolor[HTML]{FFCCC9}bad} & {\cellcolor[HTML]{9AFF99}good}\\
		\hline
	\end{tabular}
\end{center}
	
	\section{The main result}
	\label{Main}
	\noindent

    This section introduces the paper's main result, which will be used in subsequent sections to identify $R$-bad spaces.
    The proof of the main theorem relies on the following lemmas:
    
	\begin{lemma}
    [\cite{BK72}, Ch. IV, 5.3]
    \label{lmm:free}
		Let $F$ be a free group.
		Then
		$$
		R_{\infty}K(F, 1) \simeq K(\widehat{F}_R, 1).
		$$
	\end{lemma}

    \begin{lemma}
    \label{lmm:aux1}
    Let $R = \Z/p\Z$ for $p$ prime or $R \subseteq \mathbb{Q}$ a subring of the rationals.
    Then, the homology group $$H_2(\widehat{F_2}_R; R)$$ is uncountable.
    \end{lemma}
    \begin{proof}
        The case $R = \Z/p\Z$ was solved in \cite{IM18} using the theory of profinite groups.
        The case $R \subseteq \mathbb{Q}$ is straightforward from previous results, since the diagram
        \begin{center}
        \begin{tikzcd}
        H_2(\widehat{F_2}_\Z; \Z) \arrow[rd] \arrow[rr] &                                    & H_2(\widehat{F_2}_\Q; \Q) \\
                                                        & H_2(\widehat{F_2}_R; R) \arrow[ru] &                          
        \end{tikzcd}
        \end{center}
        commutes and, as proved in \cite{IM19}, the image of the top map is uncountable.
    \end{proof}
	
	Applying the lemmas we are able to prove the main result of this note:
	
	\begin{theorem}
		\label{th:main}
		Let $R = \Z / p\Z$ for $p$ prime or $R \subseteq \mathbb{Q}$ a subring of the rationals.
		Let $X$ be a connected space such that $H_2(X; \Z)$ is countable and such that there exists a surjective homomorphism $h: \pi_1(X) \twoheadrightarrow F_2$.
		Then $X$ is $R$-bad.
	\end{theorem}
	\begin{proof}
		Let $\{a, b\}$ be a system of generators of $F_2$, and let $x, y \in \pi_1(X)$ be such that $h(x) = a$ and $h(y) = b$, which exist because $h$ is surjective.
		Let $g : F_2 \rightarrow \pi_1(X)$ be the unique homomorphism such that $f(a) = x$ and $f(b) = y$.
		Clearly $h \circ g = id_{F_2}$.
		
		Now, there exist two maps $G : S^1 \vee S^1 \rightarrow X$ and $H : X \rightarrow S^1 \vee S^1$ such that $g$ and $h$ are the homomorphisms induced by $G$ and $H$, respectively, at the level of fundamental groups: the map
		$G$ identifies two representatives of the elements $x$ and $y$ in $\pi_1(X)$, where a common base point is chosen; while $H$ is the composition $X \rightarrow K(\pi_1(X),1) \rightarrow K(F_2,1)$, being the first map the first Postnikov piece, and the second the map induced by $h$ at the level of Eilenberg-MacLane spaces.
		Since the composite $H\circ G$ induces an automorphism of the fundamental group of $S^1 \vee S^1$ it is a weak equivalence and hence a homotopy equivalence according to the Whitehead theorem.
		
		Fix now a ring $R$ as in the statement.
		According to Lemma \ref{lmm:free}, the space $R_{\infty} K(F_2,1)$ is homotopically equivalent to $K(\widehat{F_2}_R,1)$, $\widehat{F_2}_R$ being the $R$-completion of $F_2$ as a group.
		For these choices of $R$, the homology group $H_2(K(\widehat{F_2}_R,1); R)$ is uncountable by Lemma \ref{lmm:aux1}.
		If we consider the maps induced by $G$ and $H$ at the level of $R$-completion and, in turn, the homomorphism induced by them in $H_2(R_{\infty} (-); R)$, the composition
		$$
		H_2(K(\widehat{F_2}_R, 1); R) \rightarrow
		H_2(R_{\infty}X; R) \rightarrow
		H_2(K(\widehat{F_2}_R, 1); R)
		$$
		is an isomorphism of abelian groups.
		This implies in particular that the first homomorphism is injective, and then $H_2(R_{\infty}X; R)$ should be uncountable.
		On the other hand, since $H_2(X; \Z)$ is countable and the ring $R$ is countable, a simple application of the universal coefficient theorem proves that $H_2(X; R)$ is also a countable group.
		Hence, by a cardinality argument, $H_2(X;R) \not\cong H_2(R_{\infty}X; R)$ and $X$ is $R$-bad.
	\end{proof}

    \subsection{Verifying the hypotheses}
	In this section we state some results that will allow us to apply the previous result in many concrete cases, the first one lets us check the condition of the epimorphism onto $F_2$.

    \begin{definition}
    \label{Def:very large}
        A \textbf{very large} group is a group $A$ such that there exists a surjection
        $$
        A \twoheadrightarrow F_2.
        $$
    \end{definition}

Observe that, as seen in the previous section, fundamental groups of compact surfaces give many examples of very large groups. More examples will be described in next section.

The following technical result will be very useful:
    
	\begin{proposition}
    \label{Prop:triv}
		Consider a finitely generated group $G = \langle S | R \rangle$ such that there exists two distinct generators $x_1, x_2 \in S$ satisfying that every word $w \in R$ can be reduced to the trivial empty word after getting rid all the occurrences of letters in $S$ distinct from $x_1$ and $x_2$ and then recursively reducing the word by cancelling a letter with an adjacent inverse. 
		Then $G$ is very large.
	\end{proposition}
	\begin{proof}
		Consider the free group $F_{n+2}$ in the generators of $S$, and the homomorphism $p: F_{n+2} \rightarrow F_2$ that takes $x_1$ and $x_2$ to the generators of $F_2$ and the rest of letters of $S$ to the trivial element.
        It is clear that $p$ is surjective, so we need to check that $p$ factors through $G$.
        We must show then that for every word $w \in R$, $p(w) = 1$, but this is exactly what the hypothesis of our proposition is telling us.
	\end{proof}
	
    Now we have a lemma that serve us to check the condition of the countability of the second homology group of our group:
    
     \begin{definition}
        The \textbf{Schur multiplier} of a group $A$ is its second homology group
        $$
        H_2(A; \Z).
        $$
    \end{definition}

	
	\begin{lemma}
    \label{lmm:count_schur}
		The Schur Multiplier of a finitely presented group is finitely generated, and thus countable.
	\end{lemma}
    \begin{proof}
        That it is countable follows simply from finitely generated since homology groups are abelian and a finitely generated abelian group has a surjection from the countable group $\Z^n$ for some positive integer $n$.

        Now, since our group is finitely presented we can take a $CW$-model of its classifying space with only finitely many $2$-cells, one for each relation in our chosen finite presentation.
        Thus, its second homology group is, furthermore, finitely generated.
    \end{proof}

	\section{Connected Closed Surfaces}
	\label{Surfaces}
	\noindent
	
	In this section, we deal with the goodness or badness of connected closed surfaces, giving a complete answer in the orientable case, and leaving just one case open in the non-orientable one. The first of these two statements answers a question by A. Murillo, which served as a motivation for this note.
	
	\begin{proposition}
		Let $X$ be a connected closed surface and let $R = \Z/p\Z$ for a prime $p$ or $R \subseteq \mathbb{Q}$ a subring of the rationals.
		Then we have:
		\begin{enumerate}
			\item If $X$ is the sphere or the torus, then it is $R$-good.
			
			\item If $X$ is non-orientable of genus 1 or 2, then it is $R$-good for $\Z /p\Z$ and for subrings of $\mathbb{Q}$ where the prime 2 is invertible.
            Otherwise, it is $R$-bad.

            \item If $X$ is orientable of genus bigger than 1 or non-orientable of genus bigger than 3, then it is $R$-bad.
		\end{enumerate}
	\end{proposition}
	
	\begin{proof}
		1. The sphere is simply-connected and the torus is the classifying space of an abelian group, hence both spaces are nilpotent and thus $R$-good, as stated in G4.\\
        
        2. The projective plane is $\Z /p\Z$-good for every prime $p$ because its fundamental group is finite (see criterion G2), and it is $R$-good if $2$ is invertible in $R \subseteq \Q$ because in this case its fundamental group is $R$-perfect (criterion G3).
        Bousfield and Kan also proved that the projective plane is $R$-bad for any other subring of the rationals, \cite[Chapter VII, Proposition 5.2]{BK72}.
        
        The Klein bottle, in turn, is $\Z /p\Z$-good for every prime $p$ because it is a virtually nilpotent space (criterion G5), and the case $R\subseteq\mathbb{Q}$ was completely solved in Bastardas' thesis, \cite[Section 4.4]{Bas03}.\\
		
		3. Let us undertake the remaining case.
        The second homology group of $X$ is countable, as it is equal to $\Z$ in the orientable case and trivial in the non-orientable case.
        
        Now, if $X$ is orientable of genus $k \geq 2$, a presentation of $\pi_1(X)$ is given by
		$$
        \langle a_1,\ldots,a_k,b_1,\ldots,b_k \textrm{ } | \textrm{ } a_1b_1a_1^{-1}b_1^{-1}\ldots a_kb_ka_k^{-1}b_k^{-1} \rangle.
        $$
		Using the terminology of Proposition \ref{Prop:triv}, if we take $\{a_1, a_2\}$ as our generators then the conditions of the proposition hold for such a presentation.
        Hence, $\pi_1(X)$ surjects over $F_2$, and by Theorem \ref{th:main} we have that $X$ is $R$-bad.
		
		On the other hand, if $X$ is non-orientable of genus $2k+1 \geq 5$, i.e. $k\geq 2$, then $\pi_1(X)$ admits a presentation
		$$\langle a_1,\ldots,a_k,b_1,\ldots,b_k,c \textrm{ } | \textrm{ } a_1b_1a_1^{-1}b_1^{-1}\ldots a_kb_ka_k^{-1}b_k^{-1}c^2 \rangle.$$
		  Take again $\{a_1, a_2\}$ as our chosen generators.
        Again Proposition \ref{Prop:triv} holds for the previous presentation with this choice of generators, and $X$ is $R$-bad by Theorem \ref{th:main}.
		{}
		Finally, consider $X$ to be non-orientable of genus $2k+2 \geq 4$, i.e. $k\geq 1$, choosing the presentation of $\pi_1(X)$ given by:
		$$\langle a_1,\ldots,a_{k},b_1,\ldots,b_{k},c,d \textrm{ } | \textrm{ } a_1b_1a_1^{-1}b_1^{-1}\ldots a_{k}b_{k}a_{k}^{-1}b_{k}^{-1}cdc^{-1}d \rangle$$
		  we apply the same argument to the generators $\{a_1, c\}$ implying that $X$ is $R$-bad.
	\end{proof}
	
	Our techniques cannot handle the remaining case of a non-orientable surface of genus 3, because in that case its fundamental group does not surject over the free group in two generators, and the techniques in \cite{Bas03} do not apply for non-virtually nilpotent spaces, as the Klein bottle is.

    \textbf{Remark}. The features of the fundamental groups of the compact surfaces that have appeared in our analysis always shed light, in a completely different context, about the Borsuk capacity of these surfaces. See \cite{MR4103983}.

	\section{Very Large Groups}
	\noindent

    In this section, we present some families of very large groups whose Schur multiplier is countable, and then by Theorem \ref{th:main} their classifying spaces are $R$-bad. We expect that these examples give a hint of the applicability of our badness criterion.

    At the end of this section we show that our arguments cannot be leveraged in general to large groups.

    \subsection{Artin Groups}
	
	We first consider the case of Artin groups, one of the main families of infinite groups.
	Recall that a group $A$ is called an \emph{Artin group} if there exists a finite set of generators and the only possible relations are of the shape $aba...bab=bab...aba$ or $aba...aba=bab...bab$ for a pair of generators $a$ and $b$, where in the previous equalities the number of letters in the left coincides with the number of letters in the right.
	An Artin group is called of \emph{even type} if all its relations are of the first kind and of \emph{odd type} if they are of the second type.
	
	\begin{proposition}
		\label{pr:eAg_epi}
        Let $R = \Z / p\Z$ for $p$ prime or $R \subseteq \mathbb{Q}$ a subring of the rationals.
		Let $A$ be an Artin group of even type with at least two distinct generators that are not involved at the same time in any defining relation.
		Then $K(A,1)$ is $R$-bad.
	\end{proposition}
	\begin{proof}
		Consider $F_2 = \left<a, b\right>$, then the function $A \to F_2$ sending one of the two distinct generators of the statement to $a$, the other one to $b$ and all of the other generators of $A$ to $1$ is a well-defined surjective morphism.
		
		On the other hand, every Artin group is finitely presented, and thus by Lemma \ref{lmm:count_schur} we have that $H_2(A; \Z)$ is finitely generated.
        Together with the fact that homology groups are abelian this implies that $H_2(A; \Z)$ is countable.
        
		Therefore, an Artin group of even type satisfies the hypothesis of Theorem \ref{th:main} above, thus $K(A,1)$ is $R$-bad.
	\end{proof}
	
	This result proves that "almost all" Artin groups of even type have a surjection onto $F_2$: the only exceptions are those groups of even type whose Artin graph is complete, i.e. the groups with exactly $n$ generators and all the possible relations.
	Notice also how the surjection given in the proof of the result doesn't hold when we have some relations of odd length involving one of the chosen generators.
    
    \begin{definition}
        A \textbf{right-angled Artin group}, \textbf{RAAG}, is an Artin group such that all the relations are commutativity relations, i.e. have length 2.
    \end{definition}

    A RAAG is usually better described by a finite simple graph $\Gamma$, and $A_\Gamma$ denotes the group corresponding to such a graph.
    The previous result can be generalized for subgroups of RAAGs:

    \begin{proposition}
    [\cite{AnM13}, Corollary 1.6]
    \label{prop:AM13}
		Any subgroup of a RAAG is either free abelian of finite rank or very large.
	\end{proposition}

    Taking into account Lemma \ref{lmm:count_schur} above we are thus looking for finitely presented non-abelian subgroups of RAAGs.
	
	
	\begin{corollary}
		Let $R = \Z / p\Z$ for $p$ prime or $R \subseteq \mathbb{Q}$ a subring of the rationals.
		Let $A$ be a non-commutative right-angled Artin group. Then $K(A,1)$ is $R$-bad.
	\end{corollary}
	\begin{proof}
		As the relations in a right-angled Artin group are all commutativity relations the non-commutative hypothesis over the group implies that the hypotheses of Proposition \ref{pr:eAg_epi} above are satisfied.
	\end{proof}

    \subsection{Bestvina-Brady Groups}
        Let $\Gamma$ be a finite simple graph, and $A_\Gamma$ its corresponding RAAG.
        The corresponding \textbf{Bestvina-Brady} group $H_\Gamma$ is the kernel of the morphism
        $$
        \varphi: A_\Gamma \to \Z,
        $$
        sending each generator to $1$.

    \begin{theorem}
        Let $\Gamma$ be a finite simple graph.
        Then $H_\Gamma$ is finitely presented if and only if the flag complex $\Delta_\Gamma$ is simply-connected.
    \end{theorem}

    Also, by the Dicks-Leary presentation, \cite{DL12}, such a Bestvina-Brady group with $\Delta_\Gamma$ simply-connected is non-abelian and thus very large by Proposition \ref{prop:AM13}, so our criteria applies.

    \subsection{Graph Braid Groups}
    Given a finite simplicial graph, its \textbf{graph braid group} is the fundamental group of the unordered configuration space of $n$-points in the graph.
    
    A graph braid group has a finite presentation as shown by Farley and Sabalka, \cite{FS12}.
    Furthermore, any non-commutative graph braid group is very large, because it embeds in a RAAG as shown by \cite[Theorem 1.1]{HW08}.
    
    \subsection{Large Groups}
    Our previous arguments about the $R$-badness of Eilenberg-MacLane spaces of very large groups don't generally extend to \emph{large} groups, as we demonstrate here.

    \begin{definition}
        A \textbf{large} group is a group with a very large subgroup of finite index.
    \end{definition}

    To show this we introduce the concept of \emph{deficiency}:

    \begin{definition}
        A group $G$ has \textbf{deficiency} $n$, for some $n > 0$, if there exists a presentation of $G$ with $k$ generators and $k-n$ relators.
    \end{definition}

    We now cite some results:

    \begin{proposition}
    [\cite{BP78}]
        A group of deficiency $2$ or greater is large.
    \end{proposition}

    \begin{proposition}
    [\cite{St83}]
        A group of deficiency $1$ such that one of the relators is a proper power is large.
    \end{proposition}


    Now, for our counterexample, consider the group $G := \Z \ast \Z/2 \cong \langle a,b | b^2 \rangle$, for which the latter holds.
    This group can be expressed as a pushout of the groups $\Z$ and $\Z/2$ through the trivial group.
    Take $R = \Z/p\Z$, $p$ odd prime, the commutation of the $R$-completion and colimits gives that
    $$
    (\Z/p\Z)_{\infty} K(G,1) = (\Z/p\Z)_{\infty} K(\Z,1)= K(\widehat{\Z}_{\Z/p\Z},1).
    $$
    This proves that $K(G,1)$ is $R$-good in this case.

    \printbibliography

    \bigskip\noindent
    
	Ramón Flores
	
	Departamento de Geometr\'{i}a y Topolog\'{i}a, Universidad de Sevilla-IMUS
	
	E-41012 Sevilla, Spain, e-mail: {\tt ramonjflores@us.es}
	
	\bigskip\noindent

    Jaime Benabent Guerrero

    Departamento de Geometr\'{i}a y Topolog\'{i}a, Universidad de Sevilla-IMUS

    E-41012 Sevilla, Spain, e-mail: {\tt jbenabent@us.es}
    
\end{document}